\documentclass[a4paper,11pt]{article}
\pdfoutput=1
\usepackage[a4paper, margin = 1.1in, bottom = 1in]{geometry}

\usepackage{hyperref,authblk,amsmath,amssymb,amsthm,color,tikz,comment}
\usepackage[font={small},width=1.14\linewidth]{caption}
\usetikzlibrary{arrows}
\usetikzlibrary{arrows.meta}
\newtheorem{theorem}{\bf Theorem}
\newtheorem{lemma}{Lemma}
\definecolor{darkgreen}{rgb}{0.,0.6,0.}

\title{The excenters of bicentric polygons are concyclic}
\author[1]{Norbert Hungerb\"uhler}
\author[2]{Clemens Pohle}
\author[3]{Yun Zhang}
\affil[1]{Department of Mathematics, ETH, Z\"urich, Switzerland}
\affil[2]{Kantonsschule Hohe Promenade, Z\"urich, Switzerland}
\affil[3]{Xi'an, Shaanxi Province, China}
\begin{document}
\maketitle

\begin{abstract}
\noindent We show that the centers of the excircles of a bicentric polygon $B$ are concyclic on a circle $E$. The center of the circumscribed circle $K$ of $B$ is the 
midpoint of the center of $E$ and the center of the inscribed circle $C$ of $B$. The radius of $E$ is given by
a simple formula in terms of the radii of $C$ and $K$ and the distance between their centers.
\end{abstract}

{\bf Key words.} bicentric polygons, excircles, Feuerbach's nine-point circle, Euler line.

{\bf Mathematics Subject Classification.} 51M04, 51M15

\section{Introduction}\footnotetext[1]{norbert.hungerbuehler@math.ethz.ch}\footnotetext[2]{clemens.pohle@kshp.ch}\footnotetext[3]{yunzhangmath@126.com}
A polygon that is inscribed in a circle $K$ and circumscribed around a circle $C$ is called a bicentric polygon.
Clearly all triangles are bicentric. If $R_K$ and $R_C$ are the radii of the circumcircle $K$ and incircle $C$ of a triangle, and
$d$ the distance between their centers, then the Chapple-Euler relation
\begin{equation}\label{eq-fuss3}
\frac1{R_K-d}+\frac1{R_K+d}=\frac1{R_C}
\end{equation}
holds (see~\cite{euler} and~\cite{chappel}).
A convex quadrilateral is bicentric if it satisfies Fuss's condition
\begin{equation}\label{eq-fuss4}
\frac1{(R_K-d)^2}+\frac1{(R_K+d)^2}=\frac1{R_C^2}
\end{equation}
(see, e.g.,~\cite[Problem 39]{doerrier}).
Fuss~\cite{fuss},  Euler~\cite{euler}, and Steiner~\cite{steiner}  found formulas relating 
the radii of $C$ and $K$ and the distance between their centers such that the circles carry a bicentric polygon
for up to 10 vertices.
Later, Jacobi~\cite{jacobi}, Richelot~\cite{richelot} and Kerawala~\cite{kerawala} used elliptic functions 
to solve the problem for a general number of vertices.
Also note that the circles $K$ and $C$ of a bicentric polygon form a Poncelet pair, i.e., every point
of $K$ is the vertex of a bicentric polygon inscribed in $K$ and circumscribed around $C$. 
Poncelet's porism holds more generally for pairs of conics, not necessarily disjoint and possibly nested
(see, e.g.,~\cite{hhponcelet} for a simple proof, and~\cite{dragovic,flatto} for a comprehensive treatment of Poncelet's porism).
Bicentric polygons have been studied from various viewpoints, we refer the interested reader to the respective literature:
see~\cite{radic2015} and the references therein, or~\cite{cheng,roitman,gg}.
In this article we would like to add a geometric property of bicentric polygons that seems to be new.

It is well known that the centers of the excircles of a convex quadrilateral are concyclic. See Figure~\ref{fig-4eck} for
a short proof. By an excircle of a convex polygon we mean a circle that is tangent to one side of the polygon 
and the extensions of the two adjacent sides\footnote{Sometimes such circles are called escribed circles.}. 
Below, we will extend the notion of excircles to general, non-convex and even self-intersecting polygons.
This is necessary so that we can state our main Theorem~\ref{thm-main}  in its full generality.
Notice that the centers of the excircles of a convex polygon with more than four vertices are
in general not concyclic. However, it turns out that bicentric polygons have exactly this 
property (see Theorem~\ref{thm-main} in Section~\ref{sec2}, and Figure~\ref{fig-0} for the case of a bicyclic pentagon).
In Section~\ref{sec-quadrilateral}, we examine the case of bicentric quadrilaterals more closely. And in 
Section~\ref{sec4}, we show that the ratio of the areas of the bicentric polygon $B$ and the polygon $P$ given by the centers of the excircles equals the ratio of the radii of the incircle of $B$ and the circumcircle of $P$.
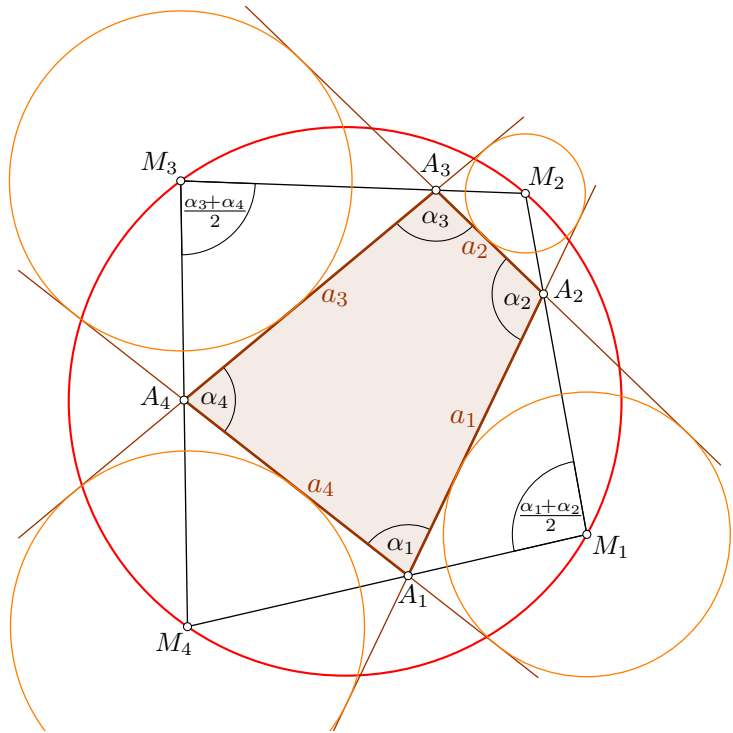
\begin{figure}[h]
\begin{center}
\begin{tikzpicture}[line cap=round,line join=round,x=18,y=18]
\definecolor{zzttqq}{rgb}{0.6,0.2,0.}
\clip(-5.050831243382393,-5.942081634663557) rectangle (10.106917109494933,9.239424944472251);

\draw[line width=0.5pt] (-1.4290423462625705,5.574480010324912) -- (5.738469853655193,5.31192082248471) -- (7.017353083179694,-1.8346304816762613) -- (-1.2886476698017661,-3.768199801020019) -- cycle;

\fill[line width=1.pt,color=zzttqq,fill=zzttqq,fill opacity=0.10000000149011612] (-1.36,0.98) -- (3.3,-2.7) -- (6.115535098078982,3.204835474831418) -- (3.88,5.38) -- cycle;
\draw [shift={(3.3,-2.7)},line width=0.4pt] (0,0) -- (64.50732779313682:1.069120181629282) arc (64.50732779313682:141.70186466650978:1.069120181629282) -- cycle;
\draw [shift={(3.88,5.38)},line width=0.4pt] (0,0) -- (-139.979999596311:1.069120181629282) arc (-139.979999596311:-44.21582439295:1.069120181629282) -- cycle;
\draw [shift={(6.115535098078982,3.204835474831418)},line width=0.4pt] (0,0) -- (135.78417560705003:1.069120181629282) arc (135.78417560705003:244.50732779313685:1.069120181629282) -- cycle;
\draw [shift={(-1.36,0.98)},line width=0.4pt] (0,0) -- (-38.29813533349022:1.069120181629282) arc (-38.29813533349022:40.020000403689025:1.069120181629282) -- cycle;
\draw [shift={(7.017353083179694,-1.8346304816762613)},line width=0.4pt] (0,0) -- (100.14575170009341:1.55) arc (100.14575170009341:193.10459622982333:1.55) -- cycle;
\draw [shift={(-1.4290423462625705,5.574480010324912)},line width=0.4pt] (0,0) -- (-89.1390674649006:1.55) arc (-89.1390674649006:-2.097911994630483:1.55) -- cycle;
\draw [line width=1.pt,color=zzttqq] (-1.36,0.98)-- (3.3,-2.7) node[midway,right] {$a_4$};
\draw [line width=1.pt,color=zzttqq] (3.3,-2.7)-- (6.115535098078982,3.204835474831418)node[midway,left,yshift=6pt,xshift=5pt] {$a_1$};
\draw [line width=1.pt,color=zzttqq] (6.115535098078982,3.204835474831418)-- (3.88,5.38)node[midway,left,xshift=4pt,yshift=-3pt] {$a_2$};
\draw [line width=1.pt,color=zzttqq] (3.88,5.38)-- (-1.36,0.98)node[midway,right,yshift=-1pt] {$a_3$};
\draw [line width=0.8pt,red] (1.99024656588744,0.9534677020284253) circle (5.748503406816389);

\draw [line width=.5pt,domain=-4.8:6,color=zzttqq] plot(\x,{(-0.438-3.68*\x)/4.66});

\draw [line width=.5pt,domain=-5.050831243382393:7.2,color=zzttqq] plot(\x,{(-27.08790183175693--5.904835474831418*\x)/2.815535098078982});

\draw [line width=.5pt,domain=-5.050831243382393:9.8,color=zzttqq] plot(\x,{(-20.46681718531902--2.175164525168582*\x)/-2.235535098078982});

\draw [line width=0.5pt,domain=-4.8:5.7,color=zzttqq] plot(\x,{(-11.1192-4.4*\x)/-5.24});

\draw [line width=0.5pt,color=orange] (-1.4290423462625705,5.574480010324912) circle (3.562942782632381);
\draw [line width=0.5pt,color=orange] (5.738469853655193,5.31192082248471) circle (1.2472346837621457);
\draw [line width=0.5pt,color=orange] (7.017353083179694,-1.8346304816762613) circle (2.9829815558231343);
\draw [line width=0.5pt,color=orange] (-1.2886476698017661,-3.768199801020019) circle (3.682149921976568);

\begin{small}
\draw [fill=white] (-1.36,0.98) circle (1.5pt) node[left,xshift=-1pt,yshift=1pt] {$A_4$};
\draw [fill=white] (3.3,-2.7) circle (1.5pt)  node[below,xshift=2pt] {$A_1$};
\draw [fill=white] (6.115535098078982,3.204835474831418) circle (1.5pt) node[right,yshift=1pt] {$A_2$};
\draw [fill=white] (3.88,5.38) circle (1.5pt) node[above,yshift=1.7pt] {$A_3$};
\draw [fill=white] (-1.2886476698017661,-3.768199801020019) circle (1.5pt) node[below,xshift=-5pt,yshift=1pt] {$M_4$};
\draw [fill=white] (7.017353083179694,-1.8346304816762613) circle (1.5pt) node[below,xshift=10pt,yshift=3pt,xshift=-1pt] {$M_1$};
\draw [fill=white] (5.738469853655193,5.31192082248471) circle (1.5pt)  node[above,xshift=8pt,yshift=-2pt] {$M_2$};
\draw[color=black] (3.549646662168704,-2.2001609989610653) node[xshift=-7pt,yshift=2pt] {$\alpha_1$};
\draw[color=black] (4.214876997404702,5.046098009859623) node[xshift=-6pt,yshift=-5pt] {$\alpha_3$};
\draw[color=black] (5.97298574052841,3.21671458796063) node[xshift=-7pt,yshift=-3pt] {$\alpha_2$};
\draw[color=black] (-0.6080429330562812,1.0784742247020662) node[xshift=-2pt,yshift=-2pt] {$\alpha_4$};
\draw [fill=white] (-1.4290423462625705,5.574480010324912) circle (1.5pt) node[above,xshift=-8pt,yshift=-1pt] {$M_3$};
\draw[color=black] (6.709490754539693,-0.7746674234553558) node[xshift=-8pt,yshift=-12pt] {$\frac{\alpha_1\!+\alpha_2}2$};
\draw[color=black] (-0.8,4.998581557342766) node[xshift=1.5pt,yshift=-1pt] {$\frac{\alpha_3\!+\alpha_4}2$};
\end{small}
\end{tikzpicture}
\caption{The centers of the orange excircles of a convex quadrilateral $A_1A_2A_3A_4$ are concyclic. The thin black lines are the external angle
bisectors of the quadrilateral. Their intersections, the points $M_1,M_2,M_3,M_4$, are the centers of the excircles. The excircle with center $M_i$
touches the side $a_i$ and the extended sides $a_{i-1}$ and $a_{i+1}$.
The sum of opposite angles, e.g.~in $M_1$ and $M_3$, equals $\pi$.
Hence, $M_1,M_2,M_3,M_4$ lie on a circle.}\label{fig-4eck}
\end{center}
\end{figure}

\section{Main result and proof}\label{sec2}
Before we state the main result it is necessary to fix the general framework. 
For the definition of a polygon, we follow~\cite[\S 3.1]{coxeter2}:
For $n\geq 3$, a polygon or $n$-gon  consists of $n$ points (the vertices) $A_1,A_2,\ldots,A_n$ and $n$ straight lines (or sides) $a_1,a_2,\ldots,$ $a_n$ such that for each $i\in\{1,\ldots,n\}$ we have
\begin{itemize}
\item $A_i$ and $A_{i+1}$ are distinct and lie on $a_i$,
\item $a_i$ and $a_{i+1}$ are distinct and meet in $A_{i+1}$.
\end{itemize} 
Here and in the sequel indices are always read cyclically.
Note that we allow our polygons to be  non-convex and even self-intersecting,
and that the sides are extended straight lines, not only the segments between two vertices.
The polygon is called cyclic if all its vertices lie on a circle (the circumcircle) $K$.
The polygon has an inscribed circle $C$ if all its sides are tangent to $C$.
The polygon is bicyclic, if it has both a circumcircle and an inscribed circle.
In every vertex of a bicyclic polygon, we have two orthogonal angle bisectors of the sides
meeting in that vertex. One of them is incident with the center of the inscribed circle and
is called internal angle bisector, the other one is called external angle bisector.
The external angle bisectors in adjacent points $A_i,A_{i+1}$ meet in the center $M_i$
of the excircle touching the sides $a_{i-1},a_i$, and $a_{i+1}$ (see Figure~\ref{fig-excircles}).

\begin{figure}[h!]
\begin{center}
\begin{tikzpicture}[line cap=round,line join=round,x=37,y=37]
\definecolor{zzttqq}{rgb}{0.6,0.2,0.}
\clip(-1,-2.4) rectangle (5.5,2);
\draw [color=zzttqq,line width=1pt,domain=-1.7177326493350782:7.380992441294493] plot(\x,{(4.7044-0.02*\x)/2.84});
\draw [color=zzttqq,line width=1pt,domain=-1.7177326493350782:4.6] plot(\x,{(7.4816-2.16*\x)/1.48});
\draw [color=zzttqq,line width=1pt,domain=-1.7177326493350782:7.380992441294493] plot(\x,{(5.4935556482988845-1.2288197580751827*\x)/-1.5374311008686283});
\draw [line width=0.8pt,color=black,domain=2.1:4.19] plot(\x,{(--2.546429180502941-0.4693569078517777*\x)/0.8830085464205983});
\draw [line width=0.8pt,color=blue,domain=1.22:2.45] plot(\x,{(-1.2964946697472843--0.8830085464205983*\x)/0.4693569078517777});
\draw [line width=0.8pt,color=black,domain=3.8:4.05] plot(\x,{(--3.8549241930404103-0.9890836022685863*\x)/-0.14735544687386778});
\draw [line width=0.8pt,color=blue,domain=1.1:4.05] plot(\x,{(--0.04857433387850996-0.14735544687386778*\x)/0.9890836022685863});
\draw [line width=0.8pt,color=orange] (4.0095858165961165,0.7525435428147303) circle (0.8756770969850698);
\draw [color=blue,line width=0.8pt] (1.3847178787656216,-0.15718730710390866) circle (1.80386991657818);
\begin{small}
\draw [fill=white] (-0.5,1.66) circle (1.5pt)  node[below] {$A_{i+2}$};
\draw [fill=white] (2.34,1.64) circle (1.5pt)  node[above,xshift=16pt,yshift=-1.5pt] {$A_{i+1}$};
\draw[,color=zzttqq] (.7,1.8) node {$a_{i+1}$};
\draw [fill=white] (3.82,-0.52) circle (1.5pt) node[below,yshift=-2.3pt] {$A_i$};
\draw[,color=zzttqq] (4.6,-1.4) node {$a_i$};
\draw [fill=white] (1.7448057648840734,-2.1786367197290515) circle (1.5pt) node[below,xshift=11pt,yshift=5pt] {$A_{i-1}$};
\draw[color=blue] (-.1,.5) node {$C$};

\draw[color=zzttqq] (3.15,-1.3723207888077527) node {$a_{i-1}$};
\draw [fill=white] (4.0095858165961165,0.7525435428147303) circle (1.5pt) node[right,yshift=3pt] {$M_i$};
\draw [fill=blue] (1.3847178787656216,-0.15718730710390866) circle (1.5pt) node[blue,below,xshift=6pt] {$M_C$};
\end{small}
\end{tikzpicture}
\caption{Internal angle bisectors (blue) meeting in the center $M_C$ of the inscribed circle $C$. External angle bisectors (black)
meeting in the center $M_i$ of the excircle (orange) that touches  the sides $a_{i-1},a_i,a_{i+1}$.}\label{fig-excircles}
\end{center}
\end{figure}
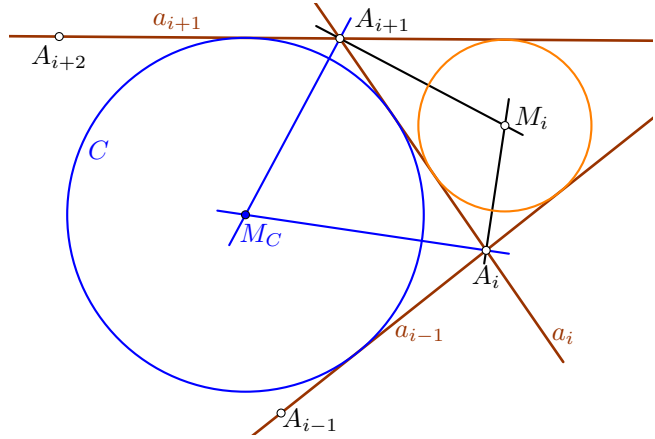

Now we can state the main result.

\begin{theorem}\label{thm-main}
Let $B$ be a bicentric polygon, inscribed in a circle $\textcolor{darkgreen}{K}$ and circumscribed around a circle $\textcolor{blue}C$.
Then the centers of the excircles of $B$ are concyclic on a circle $\textcolor{red}E$. The center of $\textcolor{darkgreen}K$ is the
midpoint of the centers of $\textcolor{blue}C$ and $\textcolor{red}E$. The radius $\textcolor{red}{R_E}$ of $\textcolor{red}E$ is given by 
\begin{equation}\label{qqq}
\textcolor{red}{R_E}=\frac{|\textcolor{darkgreen}{R_K^{\textcolor{black}2}}-d^2|}{\textcolor{blue}{R_C}},
\end{equation}
where $\textcolor{darkgreen}{R_K},\textcolor{blue}{R_C}$ are the radii of $\textcolor{darkgreen}K$ and $\textcolor{blue}C$, respectively, and $d$ is the distance between their centers.
\end{theorem}
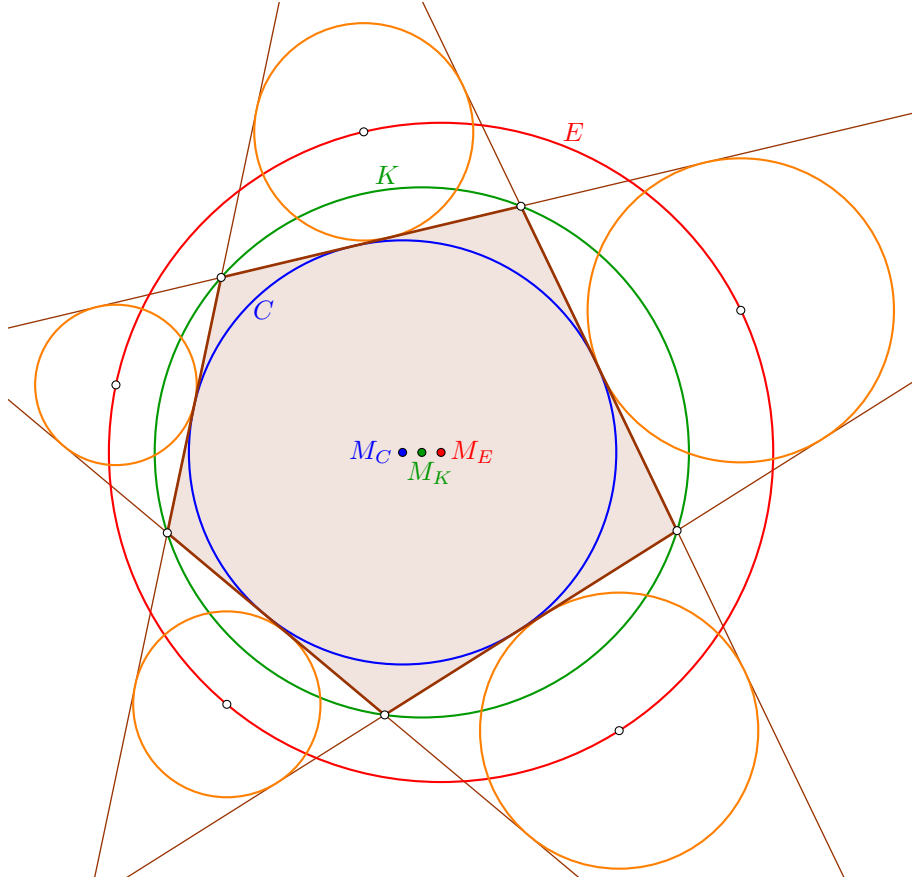
\begin{figure}[h]
\begin{center}
\definecolor{zzttqq}{rgb}{0.6,0.2,0.}
\begin{tikzpicture}[line cap=round,line join=round,x=80,y=80]
\clip(-1.846512100093996,-1.9981200984078547) rectangle (2.383849931144199,2.120570885421884);
\fill[line width=0.5pt,color=zzttqq,fill=zzttqq,fill opacity=0.12999999523162842] (-0.8490145317877361,0.8250510681914556) -- (0.5538076063389463,1.1607770965586732) -- (1.2843633057099306,-0.36884994060498993) -- (-0.0837148276035754,-1.2378672409631384) -- (-1.1010332538417709,-0.3793230193769551) -- cycle;
\draw [line width=0.8pt,color=blue] (0.,0.) circle (1.);
\draw [line width=0.8pt,color=darkgreen] (0.0900226,0.) circle (1.25);
\draw [line width=0.5pt,color=zzttqq,domain=-1.846512100093996:2.383849931144199] plot(\x,{(-0.8863630067519995--0.7998228400258824*\x)/-0.3819984349593294});
\draw [line width=0.5pt,color=zzttqq,domain=-1.846512100093996:2.383849931144199] plot(\x,{(--0.8863630067519999-0.47525222196983*\x)/-0.7481809308256826});
\draw [line width=0.5pt,color=zzttqq,domain=-1.846512100093996:2.383849931144199] plot(\x,{(--0.7343864640707839--0.47364307295978425*\x)/-0.56123588449742});
\draw [line width=0.5pt,color=zzttqq,domain=-1.846512100093996:2.383849931144199] plot(\x,{(--0.5967915708977861--0.5841397881301917*\x)/0.1222329211704527});
\draw [line width=0.5pt,color=zzttqq,domain=-1.846512100093996:2.383849931144199] plot(\x,{(-1.4424361803271903-0.33572602836721765*\x)/-1.4028221381266825});
\draw [line width=0.8pt,color=red] (0.18004519880737985,0) circle (1.5543959172430162);
\draw [line width=0.8pt,color=orange] (-0.8224635629008753,-1.1879069374901632) circle (0.43827565153887377);
\draw [line width=0.8pt,color=orange] (-1.341398044235864,0.31836632125756653) circle (0.37816761314254904);
\draw [line width=0.8pt,color=orange] (-0.18173940160949478,1.5117071440389551) circle (0.5124904407018981);
\draw [line width=0.8pt,color=orange] (1.5826774735297282,0.6699024856329198) circle (0.7168623704040236);
\draw [line width=0.8pt,color=orange] (1.0134849020563135,-1.3120689672675692) circle (0.6509329944117488);
\draw [line width=1.pt,color=zzttqq] (-0.8490145317877361,0.8250510681914556)-- (0.5538076063389463,1.1607770965586732);
\draw [line width=1.pt,color=zzttqq] (0.5538076063389463,1.1607770965586732)-- (1.2843633057099306,-0.36884994060498993);
\draw [line width=1.pt,color=zzttqq] (1.2843633057099306,-0.36884994060498993)-- (-0.0837148276035754,-1.2378672409631384);
\draw [line width=1.pt,color=zzttqq] (-0.0837148276035754,-1.2378672409631384)-- (-1.1010332538417709,-0.3793230193769551);
\draw [line width=1.pt,color=zzttqq] (-1.1010332538417709,-0.3793230193769551)-- (-0.8490145317877361,0.8250510681914556);
\begin{small}
\draw [fill=blue] (0.,0.) circle (1.5pt) node[left,blue,xshift=-0pt] {$M_C$};
\draw[color=blue] (-0.65,0.6721317116827177) node {$C$};
\draw [fill=darkgreen] (0.0900226,0.) circle (1.5pt) node[below,darkgreen,xshift=3pt] {$M_K$};
\draw[color=darkgreen] (-0.07291311184195724,1.31) node {$K$};
\draw[color=red] (.8,1.51) node {$E$};

\draw [fill=white] (1.2843633057099306,-0.36884994060498993) circle (1.5pt);
\draw [fill=white] (-0.0837148276035754,-1.2378672409631384) circle (1.5pt);
\draw [fill=white] (-1.1010332538417709,-0.3793230193769551) circle (1.5pt);
\draw [fill=white] (0.5538076063389463,1.1607770965586732) circle (1.5pt);
\draw [fill=white] (-0.8490145317877361,0.8250510681914556) circle (1.5pt);
\draw [fill=white] (1.0134849020563135,-1.3120689672675692) circle (1.5pt);
\draw [fill=white] (-0.8224635629008753,-1.1879069374901632) circle (1.5pt);
\draw [fill=white] (-1.341398044235864,0.31836632125756653) circle (1.5pt);
\draw [fill=white] (-0.18173940160949478,1.5117071440389551) circle (1.5pt);
\draw [fill=white] (1.5826774735297282,0.6699024856329198) circle (1.5pt);
\draw [fill=red] (0.18004519880737985,0) circle (1.5pt) node[right,red,xshift=0pt] {$M_E$};
\end{small}
\end{tikzpicture}

\caption{Illustration for the theorem. The centers of the orange \textcolor{orange}{excircles} of the brown \textcolor{zzttqq}{bicentric pentagon} are concyclic on the red circle $\textcolor{red}E$.
The center $\textcolor{darkgreen}{M_K}$ of the circumcircle $\textcolor{darkgreen}K$ is the midpoint of the center $\textcolor{red}{M_E}$ of $\textcolor{red}E$ and the center $\textcolor{blue}{M_C}$ of the inscribed circle $\textcolor{blue}C$.}\label{fig-0}
\end{center}
\end{figure}
In the simplest case, consider a triangle $A_1A_2A_3$ with incircle $C$, circumcircle $K$, and  centers $M_1,M_2,M_3$ of its excircles.
Then, the center of $C$ is the orthocenter of the triangle $M_1M_2M_3$, and the center of $K$ is the center of Feuerbach's nine-point circle of the triangle $M_1M_2M_3$.
Hence the theorem reduces in this case to the well known properties of the Euler line. In particular, combining~(\ref{qqq}) and~(\ref{eq-fuss3}) we get
\begin{eqnarray*}
R_E^2 &\overset{(\ref{qqq})}{=}& \frac{(R_K^2-d^2)^2}{R_C^2}\overset{(\ref{eq-fuss3})}{=}\\
&=&  \frac{(R_K^2-d^2)^2 (2R_K)^2}{(R_k-d)^2(R_K+d)^2}=4R_K^2.
\end{eqnarray*}
This is the fact that the radius of the Feuerbach circle is half the radius of the circumcircle of a triangle.

Similarly, for convex quadrilaterals, we get by~(\ref{qqq}) and~(\ref{eq-fuss4}) a simple formula which directly links $R_E$ to $R_K$ and $d$:
\begin{eqnarray}
R_E^2 &\overset{(\ref{qqq})}{=}& \frac{(R_K^2-d^2)^2}{R_C^2}\overset{(\ref{eq-fuss4})}{=}\notag\\
&=&  \frac{(R_K^2-d^2)^2 (2R_K^2+2d^2)}{(R_k-d)^2(R_K+d)^2}=2(R_K^2+d^2).\label{eq-quadrilateral}
\end{eqnarray}
The case of a bicentric convex quadrilateral will be addressed again in Section~\ref{sec-quadrilateral}.

In the proof of Theorem~\ref{thm-main} we need the following lemma which is of interest in itself.
\begin{lemma}\label{lem-main}
Let $\textcolor{blue}{C}$ be a circle with radius $\textcolor{blue}{R_C}$ and center $\textcolor{blue}{M_C}$,  $\textcolor{darkgreen}{K}$ a  
circle with radius $\textcolor{darkgreen}{R_K}$ and center $\textcolor{darkgreen}{M_K}$, and $d$ the distance between $\textcolor{blue}{M_C}$ and $\textcolor{darkgreen}{M_K}$.
Let $t$ be a tangent to $\textcolor{blue}{C}$ which
intersects $\textcolor{darkgreen}{K}$ in the points $P_1,P_2$. Then the center $\textcolor{magenta}{M_D}$ of the circumcircle $\textcolor{magenta}{D}$
of the triangle $M_CP_1P_2$ lies on a circle $F$ with center $\textcolor{darkgreen}{M_K}$ and radius
\begin{equation}\label{eq-lemma}
\textcolor{black}{R_F}=\frac{|\textcolor{darkgreen}{R_K^{\textcolor{black}2}}-d^2|}{2\textcolor{blue}{R_C}},
\end{equation}
independent of $t$ (see~Figure~\ref{fig-lemma}).
\end{lemma}
\begin{figure}[h!]
\begin{center}
\begin{tikzpicture}[line cap=round,line join=round,x=8,y=8]
\clip(-8.695872218020433,-17.083027624381664) rectangle (25.005335080526162,16.792196235086926);
\draw [line width=0.8pt,color=blue] (0.,0.) circle (2.8534923953730567);
\draw [line width=0.8pt,color=darkgreen] (6.559579622801859,0.) circle (11.329823578463174);
\draw [line width=0.5pt,domain=-8.695872218020433:25.005335080526162] plot(\x,{(--8.142418850451865-2.1150044634225815*\x)/1.9155090629267255});
\draw [line width=0.8pt,color=magenta] (17.642757312573863,10.037769507318428) circle (20.298366985361906);
\draw [line width=.8pt,{<[scale=1.1]}-{>[scale=1.1]},shorten >=1.5pt,shorten <=1.5pt] (17.642757312573863,10.037769507318428)-- (6.559579622801859,0.) node[midway,yshift=4pt,xshift=-6.7pt]{\small$R_F$};
\draw [line width=0.5pt] (6.559579622801859,0.) circle (14.953047996482574);
\draw [line width=.5pt] (6.558246940554842,0.)-- (5.070229848048584,-1.3476130370928956);
\draw [line width=.5pt] (0,0)-- (2.1150044634225815,1.9155090629267255);
\begin{small}
\draw [fill=blue] (0.,0.) circle (1.5pt) node[below,xshift=-6pt,yshift=2pt] {\textcolor{blue}{$M_C$}};
\draw[color=blue] (-2.5182843052920125,-2.001592344481265) node[yshift=-5pt] {$C$};
\draw [fill=darkgreen] (6.559579622801859,0.) circle (1.5pt) node[below,xshift=7.8pt,yshift=1pt] {\textcolor{darkgreen}{$M_K$}};
\draw[color=darkgreen] (0.6720193116099887,9.135467554522108) node[xshift=7pt] {$K$};
\draw [fill=white] (2.1150044634225815,1.9155090629267255) circle (1.5pt) node[above,xshift=4pt,yshift=-1pt] {$P$};
\draw [fill=white] (-2.414179968138258,6.916396541777396) circle (1.5pt) node[left] {$P_1$};
\draw [fill=white] (12.556005962803239,-9.612896206036856) circle (1.5pt) node[below] {$P_2$};
\draw[color=black] (-7.6,-7) node {$F$};
\draw [fill=magenta] (17.642757312573863,10.037769507318428) circle (1.5pt) node[right,xshift=-1pt,yshift=3pt] {\textcolor{magenta}{$M_D$}};
\draw[color=magenta] (5,-7) node {$D$};
\draw[color=black] (-7.6,11.6) node {$t$};
\draw [fill=white] (5.070229848048584,-1.3476130370928956) circle (1.5pt)  node[below,xshift=-2pt] {$Q$};
\end{small}
\end{tikzpicture}
\caption{Illustration for the lemma.}\label{fig-lemma}
\end{center}
\end{figure}
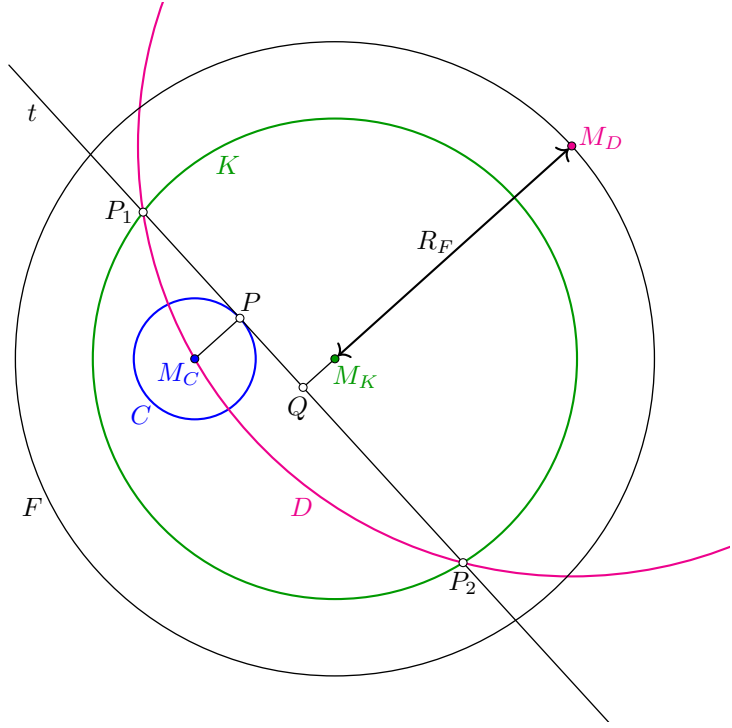
{\em Remark.} If $d<R_K=R_F$ in~(\ref{eq-lemma}) then $C$ and $F$
form a Poncelet pair for triangles. This follows by comparing 
(\ref{eq-fuss3}) and~(\ref{eq-lemma}).
\begin{proof}[Proof of Lemma~\ref{lem-main}]
We assume that $d<R_K$ in the following arguments. The case  $d>R_K$ is similar, and the case $d=R_K$ is trivial.

Let $Q$ be the intersection of the line  $M_DM_K$ with  $t$, and
$R_D$ the radius of $\textcolor{black}D$.
In the right triangles $P_1QM_D$ (used in the first line) and $P_1QM_K$ (used in the second line) we have
\begin{eqnarray}\label{eq-1}
R_D^2 &=& | P_1M_D |^2 = |P_1Q|^2 + |M_DQ|^2 \notag\\&=& 
(R_K^2 - |M_KQ|^2 ) + (|M_DM_K| \pm |M_KQ| )^2 \notag\\&=& R_K^2 + |M_DM_K|^2 \pm 2 |M_DM_K| |M_KQ|.
\end{eqnarray}
Here, depending on which side of $t$ the point $M_K$ lies, one chooses the sign in the term $ \pm |M_KQ|$
in the second line.

Similarly, if we denote the  point at which the tangent $t$ touches the circle $C$ by $P$, we find
\begin{eqnarray}\label{eq-2}
R_D^2 &=& |M_CM_D|^2 = |PQ|^2 + (R_C + |M_DM_K|\pm |M_KQ| )^2 \notag\\&=& 
(d^2 - (R_C \pm |M_KQ|)^2 ) + (R_C + |M_DM_K|\pm |M_KQ| )^2 \notag\\&=& d^2 + |M_DM_K|^2 + 2 |M_DM_K| (R_C \pm |M_KQ|).
\end{eqnarray}
Equating the terms in~(\ref{eq-1}) and~(\ref{eq-2}) and simplifying gives
\begin{equation}\label{eq-77}
R_K^2 = d^2 + 2|M_DM_K|R_C.
\end{equation} 
Hence $|M_DM_K|=:R_F$ does not depend on $t$  and~(\ref{eq-lemma}) follows from~(\ref{eq-77}).
\end{proof}
Now we are ready for the proof of  the Theorem.
\begin{proof}
Consider one side of the bicentric polygon, say $P_1P_2$, as indicated in Figure~\ref{fig-main-proof}.
The black lines in the vertices of the brown bicentric polygon are the external angle bisectors, and their
intersections the centers of the excircles of the bicentric polygon, e.g., the point $N$. The dotted internal
angle bisectors meet in $M_C$ and are orthogonal to the external angle bisectors. Hence the
circumcircle $D$ of the triangle $M_CP_1P_2$ is the circle of Thales over the segment $M_CN$.
According to  Lemma~\ref{lem-main} its center $M_D$ lies on the circle $F$ around $M_K$ with a fixed
radius $R_F$ given by~(\ref{eq-lemma}), independent of the side of the bicentric polygon. It follows by a homothety with center $M_C$
and factor $2$ that $N$ lies on a circle $E$ around the point $M_E$ with radius $2R_F$. This finishes the proof.
\begin{figure}[h]
\begin{center}
\definecolor{zzttqq}{rgb}{0.6,0.2,0.}
\begin{tikzpicture}[line cap=round,line join=round,x=20,y=20]
\clip(-6,-8.144125359216853) rectangle (9.4,7.8);
\fill[line width=1.pt,color=zzttqq,fill=zzttqq,fill opacity=0.10000000149011612] (4.657105909364631,-1.8197155131945828) -- (-2.1409516095630856,4.518442896121321) -- (-4.977857313424909,0.47003889964783224) -- (-3.649739321167288,-3.417514138599188) -- cycle;
\draw [line width=0.8pt,color=darkgreen] (0.,0.) circle (5.);
\draw [line width=0.8pt,color=blue] (-1.6666666666666656,0.) circle (2.98142396999972);
\draw [line width=0.5pt,domain=-7.835254730671592:12.317605892292885] plot(\x,{(-4.99451002067924-0.9900742016634533*\x)/-0.1405456338719762});
\draw [line width=0.5pt,domain=-7.835254730671592:12.317605892292885] plot(\x,{(-4.717254935256269-0.10439293420466805*\x)/-0.9945361307102623});
\draw [line width=0.5pt,domain=-7.835254730671592:12.317605892292885] plot(\x,{(-4.787685254921396-0.5018912592076366*\x)/0.8649307278221608});
\draw [line width=0.5pt,domain=-7.835254730671592:12.317605892292885] plot(\x,{(--4.978712365322759-0.9610034620033334*\x)/-0.27653633760793106});
\draw [line width=0.8pt,red] (1.6666666666666674,0.) circle (7.4535599249993);
\draw [line width=1.pt,color=zzttqq] (4.657105909364631,-1.8197155131945828)-- (-2.1409516095630856,4.518442896121321);
\draw [line width=1.pt,color=zzttqq] (-2.1409516095630856,4.518442896121321)-- (-4.977857313424909,0.47003889964783224);
\draw [line width=1.pt,color=zzttqq] (-4.977857313424909,0.47003889964783224)-- (-3.649739321167288,-3.417514138599188);
\draw [line width=1.pt,color=zzttqq] (-3.649739321167288,-3.417514138599188)-- (4.657105909364631,-1.8197155131945828);
\draw [line width=0.8pt,color=magenta] (2.5414148640841256,2.7258208630614655) circle (5.013785949454471);
\draw [line width=0.5pt] (-1.6666666666666656,0.)-- (1.6666666666666674,0.);
\draw [line width=0.5pt] (0.,0.)-- (2.5414148640841256,2.7258208630614655);
\draw [line width=0.5pt] (1.6666666666666674,0.)-- (6.749496394834917,5.45164172612293);
\draw [line width=0.5pt] (-1.6666666666666656,0.)-- (6.749496394834917,5.45164172612293);
\draw [line width=0.5pt] (0.,0.) circle (3.7267799624996494);

\draw [shift={(-2.1409516095630856,4.518442896121321)},line width=0.5pt]  (-84.00780796054876:0.6282063785213368) arc (-84.00780796054876:5.992192039451243:0.6282063785213368);
\draw [shift={(4.657105909364631,-1.8197155131945828)},line width=0.5pt]  (73.94640920529456:0.6282063785213368) arc (73.94640920529456:163.94640920529454:0.6282063785213368) ;
\fill[line width=1pt] (-1.853802097640709,4.285849007600053) circle (0.04);
\fill[line width=1pt] (4.4782550563186545,-1.4963470986278071) circle (0.04);

\draw [color=blue, line width=0.5pt,dotted] (-1.6666666666666656,0.)-- (-2.1409516095630856,4.518442896121321);
\draw [color=blue, line width=0.5pt,dotted] (-1.6666666666666656,0.)-- (4.657105909364631,-1.8197155131945828);

\begin{small}
\draw [fill=darkgreen] (0.,0.) circle (1.5pt)  node[below] {\textcolor{darkgreen}{$M_K$}};
\draw[color=darkgreen] (-0.8495998015143289,5.211542248146743) node {\textcolor{darkgreen}{$K$}};
\draw [fill=blue] (-1.6666666666666656,0.) circle (1.5pt) node[below,xshift=-4pt] {\textcolor{blue}{$M_C$}};
\draw[color=blue] (-1.2,2.66) node {\textcolor{blue}{$C$}};
\draw [fill=white] (-2.1409516095630856,4.518442896121321) circle (1.5pt) node[above,xshift=-4pt]{$P_1$};
\draw [fill=white] (4.657105909364631,-1.8197155131945828) circle (1.5pt) node[below,xshift=4pt]{$P_2$};
\draw [fill=white] (-4.977857313424909,0.47003889964783224) circle (1.5pt);
\draw [fill=white] (-3.649739321167288,-3.417514138599188) circle (1.5pt);
\draw [fill=white] (-4.43738551204016,4.277394365169058) circle (1.5pt);
\draw [fill=white] (-5.386641946806487,-2.4096458575195197) circle (1.5pt);
\draw [fill=white] (3.0745310640117274,-7.319390233772468) circle (1.5pt);
\draw [fill=magenta] (6.749496394834917,5.45164172612293) circle (1.5pt) node[right,yshift=7pt,xshift=0pt] {\textcolor{magenta}{$N$}};
\draw[color=red] (-2.6839624267966316,6.493083260330267) node {$E$};
\draw [fill=magenta] (1.6666666666666674,0.) circle (1.5pt) node[below,xshift=2pt] {\textcolor{red}{$M_E$}};
\draw[color=magenta] (-1.6,6) node {$D$};
\draw [fill=magenta] (2.5414148640841256,2.7258208630614655) circle (1.5pt) node[above,yshift=2.5pt,xshift=2pt] {\textcolor{magenta}{$M_D$}};
\draw[color=black] (-0.045495637007018044,3.47) node {$F$};
\end{small}
\end{tikzpicture}
\caption{Proof of the theorem.}\label{fig-main-proof}
\end{center}
\end{figure}
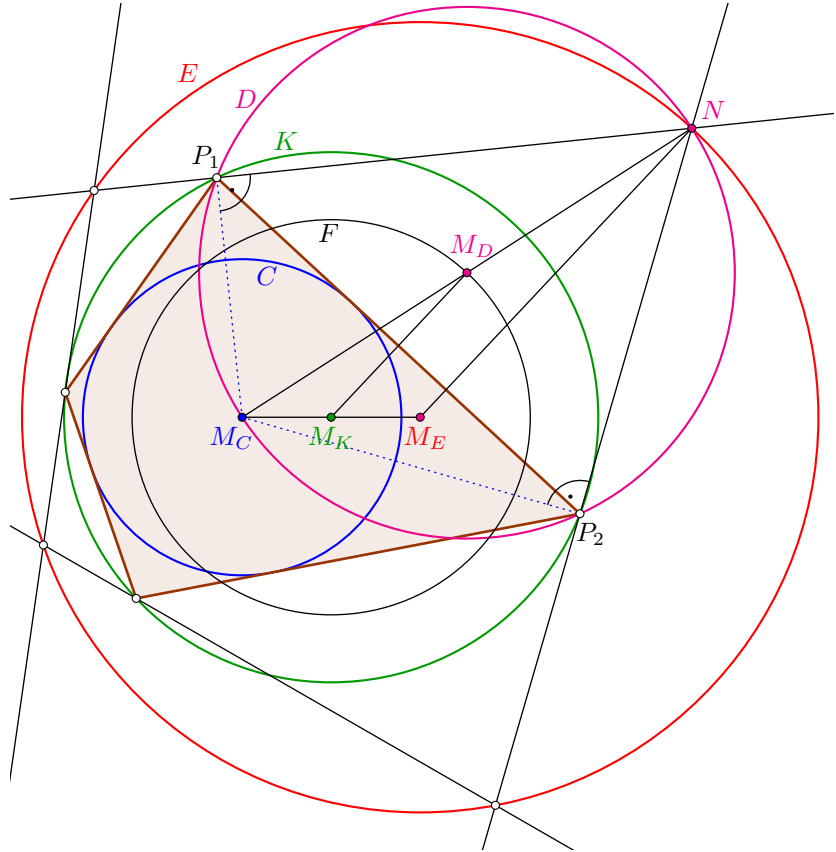
\end{proof}
Notice that it is not necessary that $C$ lies inside $K$. Figure~\ref{fig-schneiden} illustrates the theorem for a bicentric pentagon
if $C$ and $K$ intersect. Figure~\ref{fig-aussen} shows the situation for a bicentric octagon  if $C$ lies outside of $K$.

As a last remark we mention the following dynamic aspect of the main theorem. 
Since the inscribed circle $C$ and the circumcircle $K$ of a bicentric polygon form a Poncelet pair, we can choose any point on $K$ as a vertex of a bicentric polygon with inscribed circle $C$ and circumcircle $K$. It follows from the main theorem that the centers of the excircles of this new  bicyclic polygon still lie on the same circle $E$. Stated differently, if we move a vertex along $K$ the centers of the excircles roll along $E$.
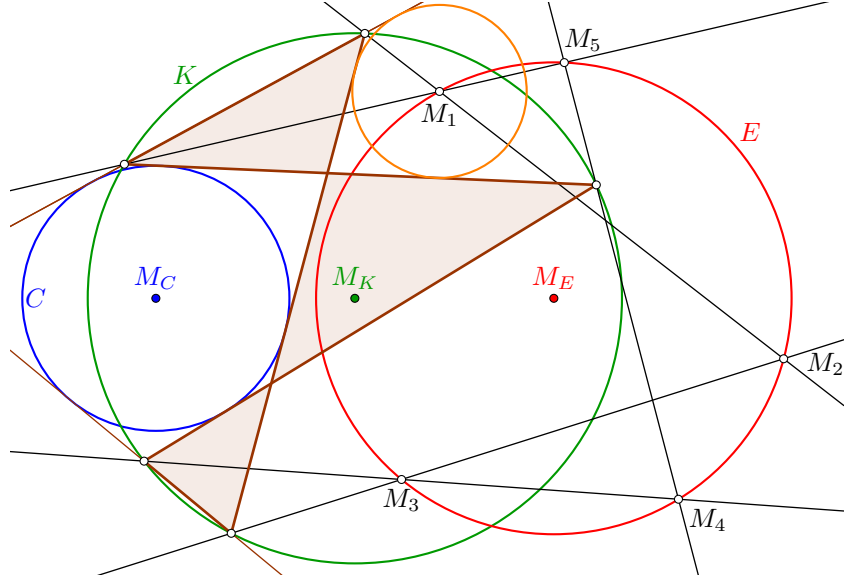
\begin{figure}[h]
\begin{center}
\definecolor{zzttqq}{rgb}{0.6,0.2,0.}
\begin{tikzpicture}[line cap=round,line join=round,x=50,y=50]
\clip(-1.0915841862347755,-2.080272980448976) rectangle (5.151859879059994,2.230195414145063);
\fill[line width=2.pt,color=zzttqq,fill=zzttqq,fill opacity=0.10000000149011612] (-0.08824088675252069,-1.2285922899062751) -- (3.2976396370511556,0.8556368588764651) -- (-0.23553709740145137,1.011351313319191) -- (1.5644636083189256,1.9986099658438903) -- (0.5644589677980355,-1.7730030849036782) -- cycle;
\draw [line width=0.8pt,color=blue] (0.,0.) circle (1.);
\draw [line width=0.8pt,color=darkgreen] (1.489910242,0.) circle (2.);
\draw [line width=0.5pt,color=zzttqq,domain=-1.0915841862347755:5.151859879059994] plot(\x,{(--1.5691251911194288--1.0050678705866263*\x)/-1.2049864899325833});
\draw [line width=0.5pt,color=zzttqq,domain=-1.0915841862347755:5.151859879059994] plot(\x,{(-0.27983781589476636-0.1345720707695448*\x)/-0.24535598825696056});
\draw [line width=0.5pt,color=zzttqq,domain=-1.0915841862347755:5.151859879059994] plot(\x,{(-2.332806931000675-1.121830723294244*\x)/-2.045356693977337});
\draw [line width=0.5pt,domain=-1.0915841862347755:5.151859879059994] plot(\x,{(--3.406837479264144-0.9679474460177145*\x)/0.2511528254823819});
\draw [line width=0.5pt,domain=-1.0915841862347755:5.151859879059994] plot(\x,{(--1.23175706476107--0.07163822357291168*\x)/-0.9974306817635585});
\draw [line width=0.5pt,domain=-1.0915841862347755:5.151859879059994] plot(\x,{(--1.860686396307982-0.3033606140820121*\x)/-0.9528758249760482});
\draw [line width=0.5pt,domain=-1.0915841862347755:5.151859879059994] plot(\x,{(--2.5381072036706382-0.616389885366694*\x)/0.787441114762008});
\draw [line width=0.5pt,domain=-1.0915841862347755:5.151859879059994] plot(\x,{(-1.038416680916073-0.22682329909575855*\x)/-0.9739359275575146});
\draw [line width=0.8pt,red] (2.9798204827220367,-9.39218910211572E-10) circle (1.7801674717780755);
\draw [line width=1.pt,color=zzttqq] (-0.08824088675252069,-1.2285922899062751)-- (3.2976396370511556,0.8556368588764651);
\draw [line width=1.pt,color=zzttqq] (3.2976396370511556,0.8556368588764651)-- (-0.23553709740145137,1.011351313319191);
\draw [line width=1.pt,color=zzttqq] (-0.23553709740145137,1.011351313319191)-- (1.5644636083189256,1.9986099658438903);
\draw [line width=1.pt,color=zzttqq] (1.5644636083189256,1.9986099658438903)-- (0.5644589677980355,-1.7730030849036782);
\draw [line width=1.pt,color=zzttqq] (0.5644589677980355,-1.7730030849036782)-- (-0.08824088675252069,-1.2285922899062751);

\draw [line width=0.8pt,orange] (2.1237502559861925,1.5608138866450778) circle (0.6528075964754018);
\begin{small}
\draw [fill=blue] (0.,0.) circle (1.5pt)  node[above] {\textcolor{blue}{$M_C$}};
\draw[color=blue] (-0.9,0) node {$C$};
\draw [fill=darkgreen] (1.489910242,0.) circle (1.5pt) node[above] {\textcolor{darkgreen}{$M_K$}};
\draw[color=darkgreen] (0.22,1.6880193113875315) node {$K$};
\draw [fill=white] (1.5644636083189256,1.9986099658438903) circle (1.5pt);
\draw [fill=white] (0.5644589677980355,-1.7730030849036782) circle (1.5pt);
\draw [fill=white] (-0.23553709740145137,1.011351313319191) circle (1.5pt);
\draw [fill=white] (3.2976396370511556,0.8556368588764651) circle (1.5pt);
\draw [fill=white] (-0.08824088675252069,-1.2285922899062751) circle (1.5pt);
\draw [fill=white] (3.9129998810226514,-1.515972440716423) circle (1.5pt) node[below,xshift=11pt] {$M_4$};
\draw [fill=white] (1.8395742220908173,-1.3670532891080696) circle (1.5pt) node[below] {$M_3$};
\draw [fill=white] (4.700532852946915,-0.4562292925589347) circle (1.5pt) node[right,xshift=3pt,yshift=-2pt,xshift=2pt] {$M_2$};
\draw [fill=white] (2.1237502559861925,1.5608138866450778) circle (1.5pt) node[below,yshift=-1pt] {$M_1$};
\draw [fill=white] (3.0582000872641077,1.7784411325168703) circle (1.5pt) node[above,xshift=7pt,yshift=1pt] {$M_5$};
\draw[color=red] (4.45,1.2569724719281274) node {$E$};
\draw [fill=red] (2.9798204827220367,-9.39218910211572E-10) circle (1.5pt) node[above] {\textcolor{red}{$M_E$}};
\end{small}
\end{tikzpicture}
\caption{The theorem with $C$ and $K$ intersecting. The thin black lines are the external angle bisectors
in the vertices of the brown bicentric pentagon. Their intersections, the centers of the excircles $M_1,\ldots,M_5$ of the pentagon, are concyclic on the 
circle $E$. To avoid overloading the figure, only the orange excircle  with center $M_1$ has been drawn.}\label{fig-schneiden}
\end{center}
\end{figure}
\begin{figure}[h]
\begin{center}
\definecolor{zzttqq}{rgb}{0.6,0.2,0.}
\begin{tikzpicture}[line cap=round,line join=round,x=22,y=22]
\clip(-2.626570401110714,-8.607278179597223) rectangle (15.742611625929325,8.643100976110444);
\fill[line width=2.pt,color=zzttqq,fill=zzttqq,fill opacity=0.10000000149011612] (3.5620274167997277,1.5722386034068734) -- (2.573536764164063,1.4168459106263847) -- (1.8030953703973298,0.5387722672002887) -- (4.866052090111511,-0.30852205047302883) -- (1.8902992589985745,-0.7377116471218877) -- (2.2941457543425443,-1.2336671145519895) -- (4.047902698110609,-1.4085042450471115) -- (1.71179367767781,0.09000958878608079) -- cycle;
\draw [line width=0.8pt,color=blue] (0.,0.) circle (1.);
\draw [line width=0.8pt,color=darkgreen] (3.3027271,0.) circle (1.5934776058467954);
\draw [line width=0.5pt,,color=zzttqq,domain=-2.626570401110714:15.742611625929325] plot(\x,{(-3.762974029152439-0.5843704372512939*\x)/-3.717322253469891});
\draw [line width=0.5pt,,color=zzttqq,domain=-2.626570401110714:15.742611625929325] plot(\x,{(--3.762974029152438-2.352687066772631*\x)/-2.936807298737275});
\draw [line width=0.5pt,,color=zzttqq,domain=-2.626570401110714:15.742611625929325] plot(\x,{(-2.76234389802623-0.42897774447080517*\x)/-2.7288316008342264});
\draw [line width=0.5pt,,color=zzttqq,domain=-2.626570401110714:15.742611625929325] plot(\x,{(--2.76234389802623-2.0763811370136334*\x)/-1.8218631081441645});
\draw [line width=0.5pt,,color=zzttqq,domain=-2.626570401110714:15.742611625929325] plot(\x,{(-1.3922425510704526--0.7517044936608022*\x)/-1.1718701614177736});
\draw [line width=0.5pt,,color=zzttqq,domain=-2.626570401110714:15.742611625929325] plot(\x,{(--1.3922425510704524-0.8704580521518386*\x)/-1.0865735596153572});
\draw [line width=0.5pt,,color=zzttqq,domain=-2.626570401110714:15.742611625929325] plot(\x,{(-4.167661270027458--2.2502183274939944*\x)/-3.5079791818505726});
\draw [line width=0.5pt,,color=zzttqq,domain=-2.626570401110714:15.742611625929325] plot(\x,{(--4.167661270027458--0.4134368457094749*\x)/-4.147103861286413});
\draw [line width=0.5pt,,color=zzttqq,domain=-2.626570401110714:15.742611625929325] plot(\x,{(-1.5941858331613739--0.42503129139530926*\x)/-1.5364819790636202});
\draw [line width=0.5pt,,color=zzttqq,domain=-2.626570401110714:15.742611625929325] plot(\x,{(--1.5941858331613734-1.1983074935875373*\x)/-1.0514217143774316});
\draw [line width=0.5pt,,color=zzttqq,domain=-2.626570401110714:15.742611625929325] plot(\x,{(-4.772174430938866--1.2723256090686268*\x)/-4.599438698777801});
\draw [line width=0.5pt,,color=zzttqq,domain=-2.626570401110714:15.742611625929325] plot(\x,{(--4.772174430938866-0.6812364680092506*\x)/-4.723300294705067});
\draw [line width=0.5pt,,color=zzttqq,domain=-2.626570401110714:15.742611625929325] plot(\x,{(-0.6395819893148111--0.49595546743010177*\x)/-0.40384649534396977});
\draw [line width=0.5pt,,domain=-2.626570401110714:15.742611625929325] plot(\x,{(--3.893581069411004-0.9148460898333288*\x)/0.40380271410265284});
\draw [line width=0.5pt,,domain=-2.626570401110714:15.742611625929325] plot(\x,{(--2.9377787205578887-0.8760145024385446*\x)/0.4822847618548082});
\draw [line width=0.5pt,,domain=-2.626570401110714:15.742611625929325] plot(\x,{(--1.8818683457278365-0.9581410806397086*\x)/0.2862964714951468});
\draw [line width=0.5pt,,domain=-2.626570401110714:15.742611625929325] plot(\x,{(--4.875822884324929-0.9979960727767959*\x)/-0.0632758936886085});
\draw [line width=0.5pt,,domain=-2.626570401110714:15.742611625929325] plot(\x,{(--2.0291500099472706-0.9315717994542089*\x)/-0.36355739912928103});
\draw [line width=0.5pt,,domain=-2.626570401110714:15.742611625929325] plot(\x,{(--2.604810797676907-0.8807339664582536*\x)/-0.4736113177772587});
\draw [line width=0.5pt,,domain=-2.626570401110714:15.742611625929325] plot(\x,{(--4.285953856691283-0.9444578344657105*\x)/-0.32863261998216275});
\draw [line width=0.5pt,,domain=-2.626570401110714:15.742611625929325] plot(\x,{(--1.7141584877166876-0.9986204250914816*\x)/0.052509490476564065});
\draw [line width=0.8pt,color=red] (6.605452792753372,-2.4672287885744944E-5) circle (8.368811262311148);
\draw [line width=1.pt,color=zzttqq] (3.5620274167997277,1.5722386034068734)-- (2.573536764164063,1.4168459106263847);
\draw [line width=1.pt,color=zzttqq] (2.573536764164063,1.4168459106263847)-- (1.8030953703973298,0.5387722672002887);
\draw [line width=1.pt,color=zzttqq] (1.8030953703973298,0.5387722672002887)-- (4.866052090111511,-0.30852205047302883);
\draw [line width=1.pt,color=zzttqq] (4.866052090111511,-0.30852205047302883)-- (1.8902992589985745,-0.7377116471218877);
\draw [line width=1.pt,color=zzttqq] (1.8902992589985745,-0.7377116471218877)-- (2.2941457543425443,-1.2336671145519895);
\draw [line width=1.pt,color=zzttqq] (2.2941457543425443,-1.2336671145519895)-- (4.047902698110609,-1.4085042450471115);
\draw [line width=1.pt,color=zzttqq] (4.047902698110609,-1.4085042450471115)-- (1.71179367767781,0.09000958878608079);
\draw [line width=1.pt,color=zzttqq] (1.71179367767781,0.09000958878608079)-- (3.5620274167997277,1.5722386034068734);
\begin{small}
\draw[color=blue] (-0.441002001292919,0.5643034982111519) node[xshift=-2pt] {$C$};
\draw[color=darkgreen] (4.216340184033097,1.1106955981656934) node[yshift=-3pt,xshift=-1pt] {$K$};
\draw [fill=white] (3.5620274167997277,1.5722386034068734) circle (1.5pt);
\draw [fill=white] (2.573536764164063,1.4168459106263847) circle (1.5pt);
\draw [fill=white] (1.71179367767781,0.09000958878608079) circle (1.5pt);
\draw [fill=white] (4.047902698110609,-1.4085042450471115) circle (1.5pt);
\draw [fill=white] (2.2941457543425443,-1.2336671145519895) circle (1.5pt);
\draw [fill=white] (1.8030953703973298,0.5387722672002887) circle (1.5pt);
\draw [fill=white] (4.866052090111511,-0.30852205047302883) circle (1.5pt);
\draw [fill=white] (1.8902992589985745,-0.7377116471218877) circle (1.5pt);
\draw [fill=white] (7.905091129161986,-8.267306095991989) circle (1.5pt);
\draw [fill=white] (0.3148210856708556,5.519541761176696) circle (1.5pt);
\draw [fill=white] (4.374210608029509,-8.065913355954093) circle (1.5pt);
\draw [fill=white] (5.410787008992121,8.283111793887949) circle (1.5pt);
\draw [fill=white] (0.11594646575151363,-5.284275761703372) circle (1.5pt);
\draw[color=red] (2.5511452127433483,6.834803311975176) node[yshift=4pt] {$E$};
\draw [fill=white] (7.435644922307615,8.327533781144007) circle (1.5pt);
\draw [fill=white] (1.373089932701097,6.531444741155022) circle (1.5pt);
\draw [fill=white] (2.086923154792657,-7.044166623949594) circle (1.5pt);
\end{small}
\end{tikzpicture}
\caption{The theorem with circle $C$ outside of $K$. The thin black lines are the external angle bisectors
in the vertices of the brown bicentric octagon. Their intersections, the centers of the excircles of the octagon, are concyclic on the 
circle $E$.}\label{fig-aussen}
\end{center}
\end{figure}
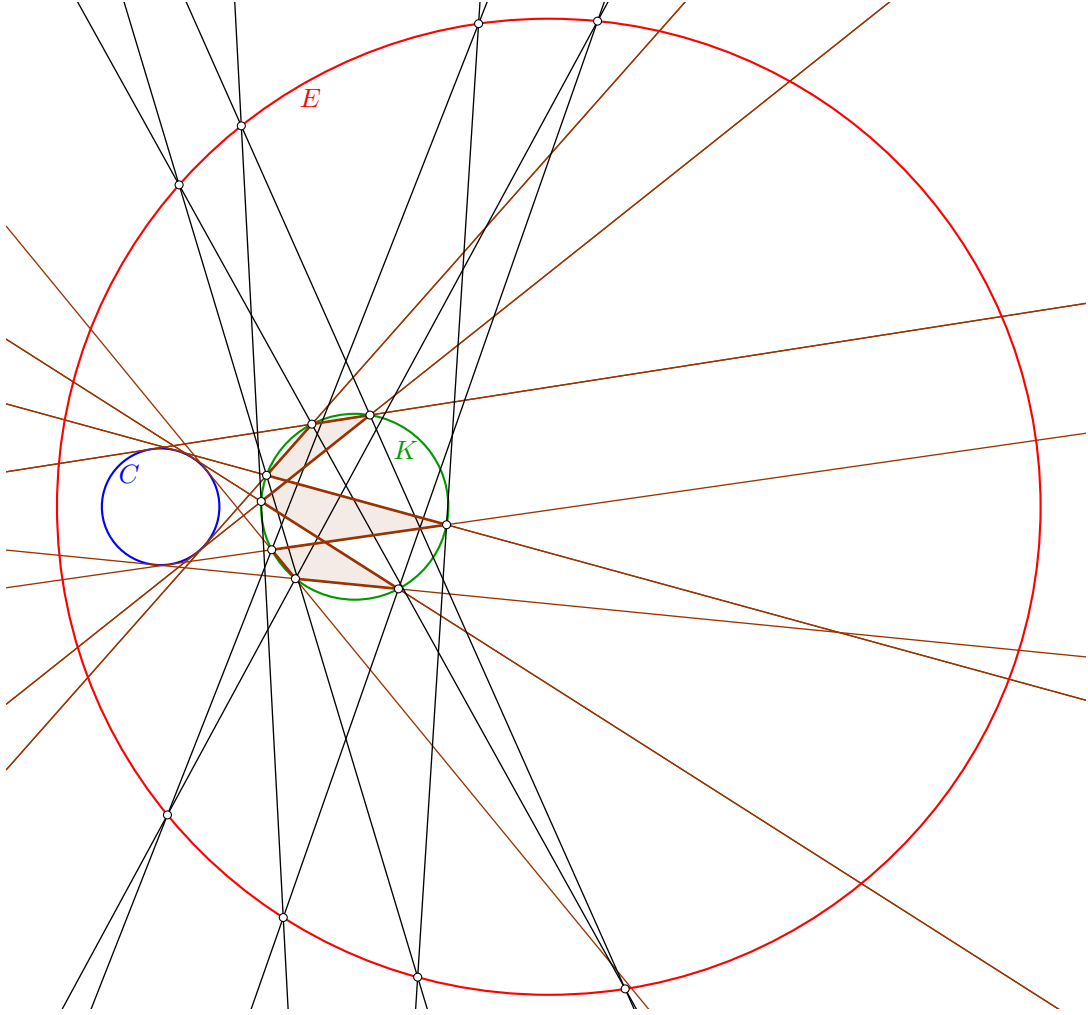
\section{Bicentric convex quadrilaterals}\label{sec-quadrilateral}
Bicentric convex quadrilaterals have been widely studied, see, e.g.,~\cite{hans,ana,martin,wu,g}.
Here we want to interpret the relation~(\ref{eq-quadrilateral}) geometrically.
We start with the following lemma which is of some independent interest.
\begin{lemma}\label{lem-diagonals}
Let $A_1A_2A_3A_4$ be a bicentric convex quadrilateral,  $M_C$  the center of its incircle, and $M_1, M_2, M_3,M_4$ 
 the centers of the excircles. Then the diagonals of the quadrilateral $M_1M_2M_3$ $M_4$ intersect perpendicularly in $M_C$.
\end{lemma}
\begin{proof}
As we have seen in the proof of Lemma~\ref{lem-main}, $A_i$ and $A_{i+1}$ lie on the  circle of Thales with diameter $M_iM_C$ (see Figure~\ref{fig-o}).
\begin{figure}[h]
\begin{center}
\definecolor{zzttqq}{rgb}{0.6,0.2,0.}
\begin{tikzpicture}[line cap=round,line join=round,x=20,y=20]
\draw[line width=1.5pt,color=zzttqq,fill=zzttqq,fill opacity=0.1] (-2.4523456637203016,5.006007881789149) -- (-0.6325976547419919,-1.7175295189389597) -- (8.084641659561921,2.499280833004442) -- (2.094572470689492,8.249391077160652) -- cycle;
\draw[line width=1pt,color=black] (-1.206544463467962,8.4897325031736) -- (8.917376485533874,7.752649374566859) -- (6.8734699732783815,-5.141484314203836) -- (-4.264301995344298,-0.060897563536627876) -- cycle;
\draw [line width=.8pt,color=darkgreen] (2.58,2.76) circle (5.510812488574031);
\draw [line width=.8pt,color=blue] (1.7490473498060923,3.5035640815949423) circle (3.662960554618243);
\draw [line width=.8pt,color=red] (3.410952650193909,2.0164359184050573) circle (7.9514054272846515);
\draw [line width=.5pt] (-4.264301995344298,-0.060897563536627876)-- (8.917376485533874,7.752649374566859);
\draw [line width=.5pt] (-1.206544463467962,8.4897325031736)-- (6.8734699732783815,-5.141484314203836);

\draw [line width=0.5pt] (0.27125144316906247,5.996648292384271) circle (2.898163146505076);
\draw [line width=0.5pt] (-1.2576273227691028,1.7213332590291572) circle (3.4952023248459696);

\draw [shift={(2.094572470689492,8.249391077160652)},line width=0.5pt,color=orange] (175.8358654325317:1.6618488341594626) arc (175.8358654325317:215.50081854516802:1.6618488341594626);
\draw [shift={(2.094572470689492,8.249391077160652)},line width=0.5pt,color=orange] (175.8358654325317:1.5107716674176932) arc (175.8358654325317:215.50081854516802:1.5107716674176932);
\draw [shift={(1.7490473498060923,3.5035640815949423)},line width=0.5pt,color=orange] (120.65768653288617:1.3596945006759238) arc (120.65768653288617:160.32263964552243:1.3596945006759238);
\draw [shift={(1.7490473498060923,3.5035640815949423)},line width=0.5pt,color=orange] (120.65768653288617:1.2086173339341544) arc (120.65768653288617:160.32263964552243:1.2086173339341544);
\draw [shift={(1.7490473498060923,3.5035640815949423)},line width=0.5pt,color=cyan] (160.32263964552243:1.6618488341594626) arc (160.32263964552243:210.6576865328861:1.6618488341594626);
\draw [shift={(1.7490473498060923,3.5035640815949423)},line width=0.5pt,color=cyan] (160.32263964552243:1.5107716674176932) arc (160.32263964552243:210.6576865328861:1.5107716674176932);
\draw [shift={(1.7490473498060923,3.5035640815949423)},line width=0.5pt,color=cyan] (160.32263964552243:1.3596945006759238) arc (160.32263964552243:210.6576865328861:1.3596945006759238);
\draw [shift={(-0.6325976547419919,-1.7175295189389597)},line width=0.5pt,color=cyan] (105.14446074587686:1.5107716674176932) arc (105.14446074587686:155.47950763324053:1.5107716674176932);
\draw [shift={(-0.6325976547419919,-1.7175295189389597)},line width=0.5pt,color=cyan] (105.14446074587686:1.3596945006759238) arc (105.14446074587686:155.47950763324053:1.3596945006759238);
\draw [shift={(-0.6325976547419919,-1.7175295189389597)},line width=.5pt,color=cyan] (105.14446074587686:1.2086173339341544) arc (105.14446074587686:155.47950763324053:1.2086173339341544);
\draw [color=blue, line width=.5pt] (2.094572470689492,8.249391077160652)-- (1.7490473498060923,3.5035640815949423);
\draw [color=blue, line width=.5pt] (1.7490473498060923,3.5035640815949423)-- (-2.4523456637203016,5.006007881789149);
\draw [color=blue, line width=.5pt] (1.7490473498060923,3.5035640815949423)-- (-0.6325976547419919,-1.7175295189389597);
\begin{small}
\draw [fill=white] (-2.4523456637203016,5.006007881789149) circle (1.5pt) node[left,yshift=5pt,xshift=1pt] {$A_1$};
\draw [fill=blue] (1.7490473498060923,3.5035640815949423) circle (1.5pt) node[right,blue,yshift=-2pt,xshift=1pt] {$M_C$};
\draw [fill=white] (-0.6325976547419919,-1.7175295189389597) circle (1.5pt) node[below] {$A_2$};
\draw [fill=white] (8.084641659561921,2.499280833004442) circle (1.5pt) node[right] {$A_3$};
\draw [fill=white] (2.094572470689492,8.249391077160652) circle (1.5pt) node[above,yshift=-1pt,xshift=3pt] {$A_4$};
\draw [fill=white] (-4.264301995344298,-0.060897563536627876) circle (1.5pt) node[left,xshift=1pt] {$M_1$};
\draw [fill=white] (6.8734699732783815,-5.141484314203836) circle (1.5pt) node[below] {$M_2$};
\draw [fill=white] (8.917376485533874,7.752649374566859) circle (1.5pt) node[above,xshift=5pt,yshift=-2pt] {$M_3$};
\draw [fill=white] (-1.206544463467962,8.4897325031736) circle (1.5pt)  node[above,xshift=-3pt] {$M_4$};
\draw (11,5.5) node[red] {$E$};
\draw (7.5,6) node[darkgreen] {$K$};
\draw (5.66,3) node[blue] {$C$};

\end{small}
\end{tikzpicture}
\caption{Proof of Lemma~\ref{lem-diagonals}.}\label{fig-o}
\end{center}
\end{figure}
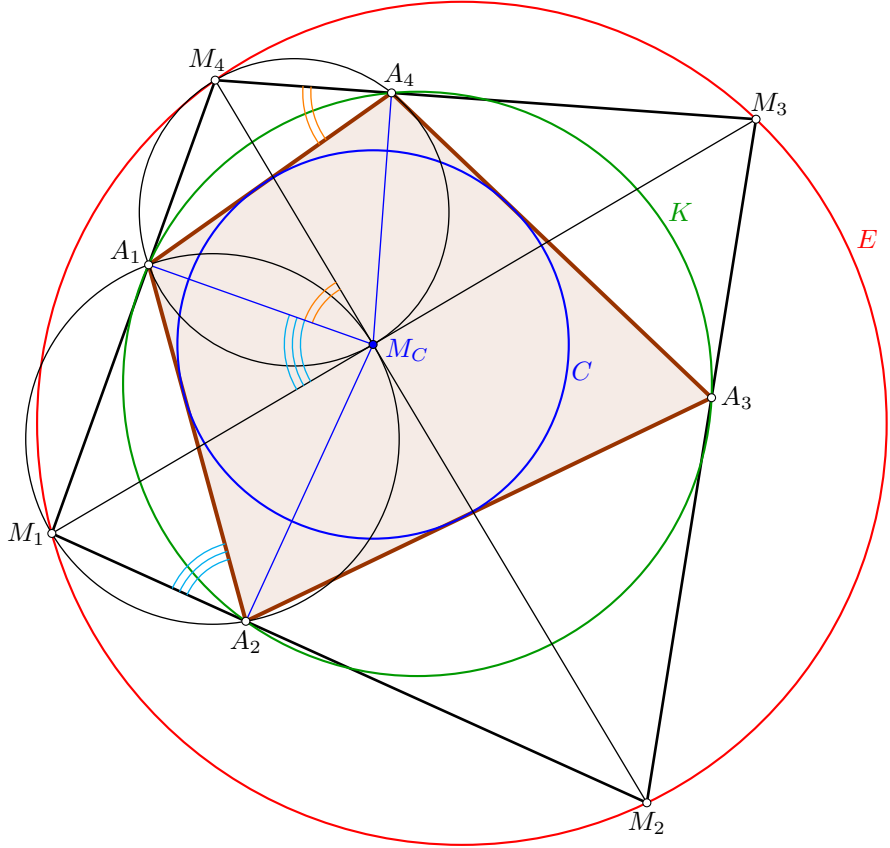
Therefore we  have
\begin{eqnarray*}
\sphericalangle M_4M_CM_1 &=& \sphericalangle M_4M_CA_1 + \sphericalangle A_1M_CM_1 \\
&=& \sphericalangle M_4A_4A_1 + \sphericalangle A_1A_2M_1 \\
&=& \bigl(\frac\pi2 - \frac{\sphericalangle A_1A_4A_3}2\bigr) + \bigl(\frac\pi2 - \frac{\sphericalangle A_3A_2A_1}2\bigr)\\ 
&=& \pi - \frac12(\sphericalangle A_1A_4A_3+\sphericalangle A_3A_2A_1) \\&=& \frac\pi2 .
\end{eqnarray*}
By shifting the indices cyclically, we get the desired result.
\end{proof}
Now we can interpret the relation~(\ref{eq-quadrilateral}) which connects the radii of the circles $E$ and $K$ with the distance $d$ of their centers geometrically:
\begin{equation}
R_E^2 =2(R_K^2+d^2).\tag{\ref{eq-quadrilateral}}
\end{equation}
\begin{proof}[Proof of~(\ref{eq-quadrilateral})]
According to Poncelet's theorem, we can move $A_1$ along $K$ until it coincides with the line $\ell$ through the points $M_C, M_K, M_E$.
In this position, the points $A_2,A_4$ lie symmetrically with respect to $\ell$ and hence $A_3$  also lies on $\ell$
 (see Figure~\ref{fig-p}).
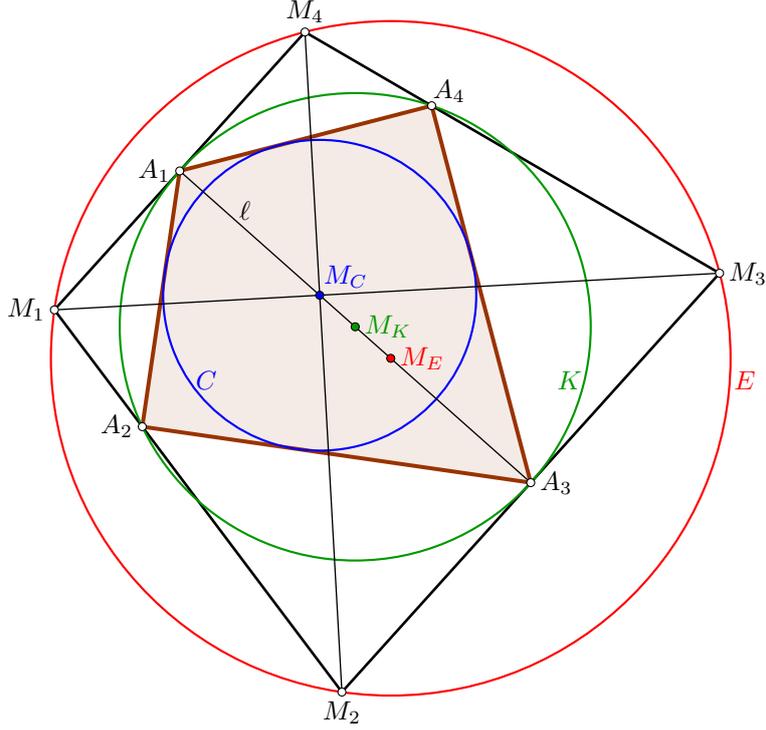
\begin{figure}[h]
\begin{center}

\definecolor{zzttqq}{rgb}{0.6,0.2,0.}
\begin{tikzpicture}[line cap=round,line join=round,x=16,y=16]
\draw[line width=1.5pt,color=zzttqq,fill=zzttqq,fill opacity=0.1] (-1.5267012696276265,6.43475284477209) -- (-2.4022232749685144,0.4049511626117529) -- (6.686739262069152,-0.9147103860714553) -- (4.369109771733129,7.972282610604568) -- cycle;
\draw[line width=1.pt] (1.4044723410511175,9.710484530538366) -- (11.104990889890622,4.022956001268159) -- (2.2684419544397585,-5.8524278233899185) -- (-4.457905185381503,3.1589872915833936) -- cycle;
\draw [line width=.8pt,color=darkgreen] (2.58,2.76) circle (5.51078976083480);
\draw [line width=.8pt,color=blue] (1.7490473498060923,3.5035640815949423) circle (3.662943563811760);
\draw [line width=.8pt,color=red] (3.410952650193909,2.0164359184050578) circle (7.95137392384789);
\draw [line width=0.5pt] (-4.457905185381503,3.1589872915833936)-- (11.104990889890622,4.022956001268159);
\draw [line width=0.5pt] (1.4044723410511175,9.710484530538366)-- (2.2684419544397585,-5.8524278233899185);

\draw [line width=0.5pt] (1.7490473498060923,3.5035640815949423)-- (-1.5267012696276265,6.43475284477209);
\draw [line width=0.5pt] (1.7490473498060923,3.5035640815949423)-- (6.686739262069152,-0.9147103860714553);
\begin{small}
\draw [fill=darkgreen] (2.58,2.76) circle (1.5pt)  node[darkgreen,right] {$M_K$};
\draw [fill=white] (-1.5267012696276265,6.43475284477209) circle (1.5pt) node[left] {$A_1$};
\draw [fill=blue] (1.7490473498060923,3.5035640815949423) circle (1.5pt) node[blue,right,yshift=7pt,xshift=-2pt] {$M_C$};
\draw [fill=white] (-2.4022232749685144,0.4049511626117529) circle (1.5pt) node[left] {$A_2$};
\draw [fill=white] (6.686739262069152,-0.9147103860714553) circle (1.5pt)  node[right] {$A_3$};
\draw [fill=white] (4.369109771733129,7.972282610604568) circle (1.5pt) node[anchor=south west,xshift=-3pt,yshift=-2pt] {$A_4$};
\draw [fill=white] (-4.457905185381503,3.1589872915833936) circle (1.5pt)  node[left] {$M_1$};
\draw [fill=white] (2.2684419544397585,-5.8524278233899185) circle (1.5pt) node[below] {$M_2$};
\draw [fill=white] (11.104990889890622,4.022956001268159) circle (1.5pt) node[right] {$M_3$};
\draw [fill=white] (1.4044723410511175,9.710484530538366) circle (1.5pt) node[above] {$M_4$};
\draw [fill=red] (3.410952650193909,2.0164359184050578) circle (1.5pt) node[red,right] {$M_E$};
\draw (11.7,1.5) node[red] {$E$};
\draw (7.6,1.5) node[darkgreen] {$K$};
\draw (-.9,1.5) node[blue] {$C$};
\draw (0,5.5) node {$\ell$};
\end{small}
\end{tikzpicture}
\caption{Proof of relation~(\ref{eq-quadrilateral}).}\label{fig-p}
\end{center}
\end{figure}
By Lemma~\ref{lem-diagonals} the triangle $M_1M_CM_4$ is now an isosceles right-angled triangle. $A_1$ is the midpoint of the segment $M_1M_4$, so the two triangles $M_1M_CA_1$ and $A_1M_CM_4$ are also isosceles right-angled triangles. Pythagoras in the triangle $A_1M_1M_E$ yields
\begin{multline*}
\qquad R_E^2 = |M_1M_E|^2 = |M_EA_1|^2 + |A_1M_1|^2 = (R_K+d)^2 + |A_1M_C|^2 =\\= (R_K+d)^2 + (R_K-d)^2=2(R_K^2+d^2),\qquad
\end{multline*}
as desired.
\end{proof}
\section{The area ratio of a bicentric convex polygon and its excenters polygon}\label{sec4}
It turns out that bicentric polygons have another nice property. This observation was originally formulated as a conjecture by a colleague.
\begin{theorem}
Let  $A_1A_2\ldots A_n$ be a bicentric convex polygon, $r_C$ the center of its incircle,
$M_1M_2\ldots M_n$ the polygon formed by the centers of the excircles, and $r_E$ the 
radius of its circumcircle. Then, we have 
$$
\frac{\operatorname{area}A_1A_2\ldots A_n}{{\operatorname{area}M_1M_2\ldots M_n}}=\frac{r_C}{r_E}.
$$
\end{theorem}
\begin{proof}
As before, we denote the center of the incircle of $A_1A_2\ldots A_n$ by $M_C$ and the center of the circumcircle of $M_1M_2\ldots M_n$ by $M_E$. Observe that the lines $M_EM_i$ are perpendicular to the sides $a_i=A_iA_{i+1}$ of the bicentric polygon. 
To see this, look again at the homothety with center $M_C$ and factor $2$ that we used in the proof of Theorem~\ref{thm-main}, and
go back to Figure~\ref{fig-main-proof}: The  line $M_KM_D$ is the perpendicular bisector of the line segment $P_1P_2$. 
By the homothety, the  line $M_KM_D$ is mapped to the  line $M_EN$, which is therefore also perpendicular to $P_1P_2$. 

Consider now the orthodiagonal quadrilateral $M_E A_i M_i A_{i+1}$ (see Figure~\ref{fig-areas} where the case $i=1$ is shown). Its area is given by
$$
\operatorname{area}M_EA_iM_iA_{i+1}=\frac{1}{2}\cdot a_i\cdot r_E.
$$
Here $a_i$ denotes the length of the side $A_iA_{i+1}$.
Similarly we have for the area of the triangle $M_CA_iA_{i+1}$
$$
\operatorname{area}M_CA_iA_{i+1}=\frac{1}{2}\cdot a_i\cdot r_C.
$$
This leads us to 
$$
\frac{\operatorname{area}{M_CA_iA_{i+1}}}{\operatorname{area}{M_EA_iM_iA_{i+1}}}=\frac{r_C}{r_E}
$$
and then, by taking the sum over $i$, to the desired result
$$
\frac{\operatorname{area}{A_1A_2\ldots A_n}}{\operatorname{area}{M_1M_2\ldots M_n}}=\frac{r_C}{r_E}.
$$
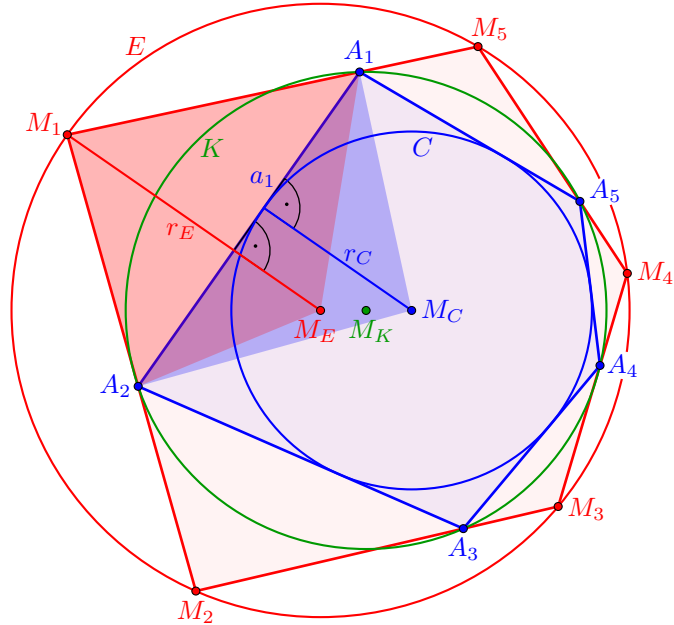
\begin{figure}[h]
\begin{center}
\begin{tikzpicture}[line cap=round,line join=round,x=45,y=45]
\draw[line width=1.pt,color=red,fill=red,fill opacity=0.05] (-1.41622,-2.352396) -- (1.598147,-1.64346) -- (2.173211,0.310477) -- (0.930974,2.212226) -- (-2.48594,1.4731574) -- cycle;
\draw[line width=1.pt,color=blue,fill=blue,fill opacity=0.05] (-0.053518,1.99928) -- (1.77874,0.914374) -- (1.9460,-0.4612) -- (0.80908,-1.829) -- (-1.896126,-0.636163) -- cycle;
\fill[line width=2.pt,color=blue,fill=blue,fill opacity=0.25] (0.3789,0.) -- (-0.053518838,1.9992838) -- (-1.896126,-0.63616) -- cycle;
\fill[line width=2.pt,color=red,fill=red,fill opacity=0.25] (-0.3789194,0.) -- (-0.053518,1.999283) -- (-2.48594,1.473157) -- (-1.89612,-0.636163) -- cycle;
\draw [shift={(-0.8504115901235729,0.8595028854624439)},line width=0.5pt] (0,0) -- (-34.95982543441039:0.3) arc (-34.95982543441039:55.04017456558961:0.3) -- cycle;
\draw [shift={(-1.0992337446139382,0.5036171734927133)},line width=0.5pt] (0,0) -- (-34.9598254344104:0.3) arc (-34.9598254344104:55.040174565589616:0.3) -- cycle;
\draw [line width=.8pt,color=blue] (0.37891944335835,0.) circle (1.5);
\draw [line width=.8pt,darkgreen] (0.,0.) circle (2.);
\draw [line width=.8pt,color=red] (-0.37891944335834926,0.) circle (2.5709467036300024);
\draw [line width=0.8pt,color=red] (-0.37891944335835,0.)-- (-2.4859491554919564,1.4731574067600954);
\draw [line width=0.8pt,color=blue] (0.37891944335835,0.)-- (-0.8504115901235729,0.8595028854624439);
\fill[line width=0.8pt] (-0.66,0.89) circle (0.015);
\fill[line width=0.8pt] (-0.92,0.54) circle (0.015);
\begin{small}
\draw [fill=darkgreen] (0.,0.) circle (1.5pt) node[below,xshift=2pt] {\textcolor{darkgreen}{$M_K$}};
\draw [fill=blue] (0.37891944335835,0.) circle (1.5pt) node[right,blue] {$M_C$};
\draw[color=blue] (0.47170835711982295,1.415726674381342) node[yshift=-2pt] {$C$};
\draw[color=darkgreen] (-1.281162594698142,1.365644647186543) node {$K$};
\draw [fill=blue] (-0.05351883897849403,1.9992838052348632) circle (1.5pt) node[above,blue] {$A_1$};

\draw [fill=white,white] (2.,1.) circle (5.2pt) ;
\draw [fill=blue] (1.7787411567601783,0.9143740466830097) circle (1.5pt) node[blue,right,yshift=4pt] {$A_5$};
\draw [fill=blue] (-1.8961264957590171,-0.6361637462797056) circle (1.5pt) node[left,blue,xshift=1pt] {$A_2$};

\draw [fill=white,white] (2.15,-0.4612333581876642) circle (5pt) ;
\draw [fill=blue] (1.94608935799334,-0.4612333581876642) circle (1.5pt) node[right,blue,xshift=-1pt] {$A_4$};
\draw [fill=red] (0.9309743115740363,2.2122263228917354) circle (1.5pt) node[above,red,xshift=4.4pt,yshift=-1pt] {$M_5$};
\draw [fill=red] (-2.4859491554919564,1.4731574067600954) circle (1.5pt) node[left,yshift=4pt,xshift=2pt,red] {$M_1$};
\draw [fill=blue] (0.8090813035881343,-1.8290400335105095) circle (1.5pt) node[below,blue] {$A_3$};
\draw [fill=red] (-1.4162217854373451,-2.3523968211216895) circle (1.5pt) node[below,red] {$M_2$};
\draw [fill=red] (2.173211230712538,0.31047701263129) circle (1.5pt) node[right,red] {\small$M_4$};
\draw [fill=red] (1.5981473973798999,-1.6434639211614244) circle (1.5pt) node[right,yshift=-1pt,red] {$M_3$};
\draw[color=red] (-1.92,2.217039109498125) node {$E$};
\draw[color=blue] (-0.86,1.1) node {$a_1$};
\draw [fill=red] (-0.37891944335835,0.) circle (1.5pt) node[below,xshift=-2pt,red] {$M_E$};
\draw[color=red] (-1.5378329840714868,0.6957975334561073) node[yshift=-1.3pt] {$r_E$};
\draw[color=blue] (-0.06041318182491645,0.45) node {$r_C$};
\end{small}
\end{tikzpicture}
\caption{The ratio of the red and the blue area is independent of the position of the starting point $A_1$ 
of the Poncelet polygon $A_1A_2\ldots A_n$ on $K$.}\label{fig-areas}
\end{center}
\end{figure}
\end{proof}

\bibliographystyle{plain}

\end{document}